\documentclass[12pt]{amsart}
\usepackage{hyperref, tikz}
\date{\today}

\theoremstyle{plain}
\newtheorem{theorem}{Theorem}

\theoremstyle{definition}
\newtheorem{definition}{Definition}
\newtheorem{lemma}{Lemma}

\theoremstyle{remark}
\newtheorem{example}{Example}

\newcommand{\raag}[1][\Gamma]{{\mathrm{RAAG}_{#1}}}
\newcommand{\racg}[1][\Gamma]{{\mathrm{RACG}_{#1}}}
\newcommand{\gm}[1][\Gamma]{{\mathrm{GM}_{#1}}}

\newcommand{\graag}[1][\Gamma]{\mathrm{RAAG}^*_{#1}(t)}
\newcommand{\gracg}[1][\Gamma]{\mathrm{RACG}^*_{#1}(t)}
\newcommand{\ggm}[1][\Gamma]{\mathrm{GM}^*_{#1}(t)}

\newcommand{\lk}{\mathrm{Lk}}

\newcommand{\m}{\mathbf m}
\renewcommand{\vec}{\mathbf}
\newcommand{\ones}{\mathbf 1'}

\begin{document}
\title[Right-Angled Groups and Monoids]{Growth in Right-Angled Groups\\and Monoids}
\begin{abstract} We derive functional relationships between spherical generating functions of graph monoids, right-angled Artin groups and right-angled Coxeter groups.
We use these relationships to express the spherical generating function of a right-angled Artin group in terms of the clique polynomial of its defining graph.
We also describe algorithms for computing the geodesic generating functions of these structures.
\end{abstract}

\author{Jayadev S. Athreya}
\thanks{J.S.A. was partially supported by NSF PI grant DMS 1069153; NSF grants DMS 1107452, 1107263, 1107367  ``RNMS: Geometric structures and representation varieties" (the GEAR Network)"; and NSF CAREER grant DMS 1351853}
\address{Department of Mathematics, University of Illinois.
1409 W. Green Street, Urbana, IL 61801, USA.}
\author{Amritanshu Prasad}
\address{The Institute of Mathematical Sciences, CIT campus Taramani, Chennai 600113, India.}
\maketitle

\tableofcontents
\section{Overview}
\label{sec:overview}
Let $\Gamma$ be a finite graph with nodes $\m = \{1,\dotsc,m\}$.
Consider the set of equivalence classes of words in the alphabet $x_1,\dotsc,x_m$, modulo the relation where two words are said to be equivalent if each can be obtained from the other by a sequence of substitutions called shuffles:
\begin{equation}
  \label{eq:1}
  u x_i x_j v \to u x_j x_i v,
\end{equation}
for some words $u$ and $v$, and some and $i,j\in \m$ which are connected by an edge in $\Gamma$.
In other words, variables commute if and only if the corresponding nodes in the graph are connected by an edge.

Concatenation of words descends to a binary operation on equivalence classes of words, resulting in a monoid $\gm$ known as the graph monoid of $\Gamma$.
The identity element of this monoid is the empty word.
Such monoids were studied by Cartier and Foata \cite{MR0239978} and therefore go by the name of Cartier-Foata monoids.

The right-angled Artin group $\raag$ associated to $\Gamma$ is the group obtained by adding to $\gm$ the inverses of all its elements.
Concretely, it may be regarded as the set of equivalence classes of words in the $2m$ symbols $x_1,x_1^{-1},\dotsc,x_m,x_m^{-1}$, where, in addition to shuffles \eqref{eq:1}, substitutions of the form
\begin{equation}
  \label{eq:2}
  ux_i x_i^{-1} v \longleftrightarrow uv \quad\text{ and }\quad ux_i^{-1} x_i v \longleftrightarrow uv,
\end{equation}
for $i\in \m$ are allowed.
Right-angled Artin groups were introduced under the name of semifree groups by Baudisch \cite{MR634562} and were called graph groups by Servatius, Droms and Servatius \cite{MR952322}.
Of late, they have attracted a lot of attention in geometric group theory (see, for example, the survey article by Charney \cite{charney2007introduction}).

The right-angled Coxeter group $\racg$ is the quotient of $\raag$ by its smallest normal subgroup containing $x_i^2$ for every $i\in \m$.
In $\racg$, each $x_i$ is equal to its inverse, and therefore, just like in $\gm$, each element is represented by a word in the alphabet $x_1,\dotsc,x_m$.
Now, in addition to shuffles (\ref{eq:1}), substitutions of the form
\begin{equation}
  \label{eq:cox-canc}
  ux_i x_i v \longleftrightarrow uv,
\end{equation}
for each $i\in \m$ are allowed.
The group $\racg$ may be viewed as a graph product of cyclic groups of order two in the sense of Green \cite{green_thesis} and Hermiller and Meier \cite{MR1314099}.

The length of an element $x$ in any of the structures $\gm$, $\raag$ and $\racg$ is the minimal length of a word that represents it, and is denoted $l(x)$.
Thus, the identity element is the only element of length $0$ in all three structures.
\begin{example}
  \label{example:small-len}
  In $\gm$ and $\racg$ there are $m$ elements of length one, namely $x_1,\dotsc,x_m$, whereas in $\raag$ there are $2m$ elements of length one, namely $x_1^{\pm 1}, \dotsc, x_m^{\pm 1}$.
\end{example}
The spherical growth functions of these structures are defined as:
\begin{align*}
  \gm(t) = \sum_{x\in \gm} t^{l(x)},\\
  \raag(t) = \sum_{x\in \raag} t^{l(x)},\\
  \racg(t) = \sum_{x\in \racg}t^{l(x)}.
\end{align*}
The theory of Tits systems gives the rationality of spherical generating functions of Coxeter groups in general (see, for example \cite[Cor.~17.1.6]{Davis}).
The rationality of spherical generating functions for right angled Artin groups follows from the results of Loeffler, Meier and Worthington \cite{MR1949695}.

An important combinatorial invariant of a graph is its clique polynomial:
\begin{equation*}
  p_\Gamma(t) = 1 + c_1 t + c_2 t^2 + \dotsb,
\end{equation*}
where $c_i$ is the number of cliques in $\Gamma$ with $i$ nodes.
The spherical growth function of $\gm$ is has a nice expression in terms of $p_\Gamma(t)$:
\begin{equation}
  \label{eq:3}
  \gm(t) = \frac 1{p_\Gamma(-t)}.
\end{equation}
This formula can be obtained by substituting $T_i = t$ for every $i$ in the Cartier-Foata identity \cite[Ch.~1, Eq. (1)]{MR0239978}.
For another elegant (and completely elementary) proof of (\ref{eq:3}) see Fisher \cite{Fishy}.

The theory of Coxeter groups gives a similar expression for right-angled Coxeter groups; see Davis \cite[Prop.~17.4.2]{Davis}:
\begin{equation}
  \label{eq:4}
  \racg(t) = \frac 1{p_\Gamma\left(\frac{-t}{1+t}\right)}.
\end{equation}

The first main result of this article is a version of the formulae \eqref{eq:3} and \eqref{eq:4} for right-angled Artin groups:
\begin{equation}
  \label{eq:5}
  \raag(t) = \frac 1{p_\Gamma\left(\frac{-2t}{1+t}\right)}.
\end{equation}
\begin{example}
  Taking $\Gamma$ to be the complete graph on $m$ nodes gives a well-known identity:
  \begin{quote}
    The number of points in $\mathbf Z^m$ which lie on the $L^1$-sphere of radius $n$ centered at the origin is the coefficient of $t^n$ in $\left(\frac{1+t}{1-t}\right)^m$.
  \end{quote}
\end{example}
\begin{example}
  Taking $\Gamma$ to be the graph with $m$ vertices and no edges gives:
  \begin{quote}
    The number of words of length $n$ in a free group on $m$ generators is the coefficient of $t^n$ in $\frac{1+t}{1-(2m-1)t}$.
  \end{quote}
\end{example}
\begin{example}
  Taking $\Gamma$ to be the square graph (Fig.~\ref{fig:square}) yields the Cartesian square $F_2 \times F_2$ of the free group on two generators.
  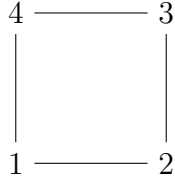
\begin{figure}[h]
    \centering
    \begin{tikzpicture}
      \path (-1, -1) node (n1) {$1$};
      \path (1, -1) node (n2) {$2$};
      \path (1, 1) node (n3) {$3$};
      \path (-1, 1) node (n4) {$4$};
      \draw
      (n1) edge (n2)
      (n2) edge (n3)
      (n3) edge (n4)
      (n4) edge (n1);
    \end{tikzpicture}
    \caption{The square graph}
    \label{fig:square}
  \end{figure}
The clique polynomial of this graph is $p_{\Gamma}(t) = (1+2t)^2$, and thus the spherical growth function of $F_2\times F_2$ with respect to its standard generators is $\left(\frac{1+t}{1-3t}\right)^2$.
\end{example}
\begin{example}
  The line graph $A_m$ (Fig.~\ref{fig:line}) with $m$ vertices has clique polynomial $1+mt+(m-1)t^2$.
  \begin{figure}[h]
    \centering
    \begin{tikzpicture}
      \path (0, 0) node (n1) {$1$};
      \path (1, 0) node (n2) {$2$};
      \path (2, 0) node (n3) {$3$};
      \path (6, 0) node (nm) {$m$};
      \draw
      (n1) edge (n2)
      (n2) edge (n3);
      \draw [dashed]
      (n3) edge (nm);
    \end{tikzpicture}
    \caption{The line graph $A_m$}
    \label{fig:line}
  \end{figure}
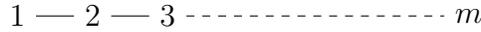
  Therefore the spherical growth function of the corresponding right-angled Artin group is
  \begin{equation*}
    \frac{(1+t)^2}{1-2(n-1)t + (2n-3)t^2}.
  \end{equation*}
\end{example}
The proof of Eq.~\eqref{eq:5} proceeds by deriving  a functional relationships between $\raag(t)$ and $\racg(t)$, namely:
\begin{equation}
  \label{eq:6}
  \raag(t) = \racg\left(\frac{2t}{1-t}\right),
\end{equation}
which allows for the deduction of \eqref{eq:6} from \eqref{eq:4}.
This functional relationship is based upon the comparison of branching rules in the sense of Prasad \cite{branching} for elements in $\raag$ and $\racg$ (see Lemma~\ref{lemma:branches} in the next section).
A similar comparison of branching rules also gives a functional relationship between $\gm$ and $\racg$:
\begin{equation}
  \label{eq:7}
  \gm(t) = \racg\left(\frac t{1-t}\right),
\end{equation}
which allows one to easily derive either of \eqref{eq:3} and \eqref{eq:4} from the other.

All the results in this article rely upon the classification of elements of our structures into types.
The type $\tau(x)$ of an element $x$ is a subset of the vertex-set of the graph $\Gamma$ (see Definition~\ref{sec:type-an-element}).
Let $M$ be any subset of the vertex-set of $\Gamma$.
A generalization of the Cartier-Foata identity due to Krattenthaler \cite[Theorem~4.1]{krattenthaler2006theory} gives rise to a refinement of the identity (\ref{eq:3}):
\begin{equation}
  \label{eq:Kratt}
  \sum_{x\in \gm, \tau(x)\subset M} t^{l(x)} = \frac{p_{\Gamma - M}(-t)}{p_\Gamma(-t)}.
\end{equation}
Here $p_{\Gamma - M}(t)$ is the clique polynomial of the induced subgraph of $\Gamma$ obtained after deleting the vertices in $M$.
The functional relations~(\ref{eq:functional1}) and (\ref{eq:functional2}) allow us to deduce similar identities for right-angled Artin groups and right-angled Coxeter groups:
\begin{gather}
  \label{eq:refined-racg}
  \sum_{x\in \racg, \tau(x)\subset M} t^{l(x)} = \frac{p_{\Gamma - M}\left(\tfrac{-t}{1+t}\right)}{p_\Gamma\left(\tfrac{-t}{1+t}\right)},\\
  \label{eq:refined-raag}
  \sum_{x\in \raag, \tau(x)\subset M} t^{l(x)} = \frac{p_{\Gamma - M}\left(\tfrac{-2t}{1+t}\right)}{p_\Gamma\left(\tfrac{-2t}{1+t}\right)}.
\end{gather}
  
A geodesic (or reduced) word in any of the structures $\gm$, $\raag$ and $\racg$ is any word which has minimal length among all the words that represent the same element.
We let $\gm^*$, $\raag^*$ and $\racg^*$ denote the sets of geodesic words in $\gm$, $\raag$ and $\racg$ respectively.
The geodesic generating functions for these structures are the functions:
\begin{align*}
  \ggm & = \sum_{x\in \gm^*}t^{l(x)},\\
  \graag & = \sum_{x\in \raag^*}t^{l(x)},\\
  \gracg & = \sum_{x\in \racg^*}t^{l(x)}.
\end{align*}
Our main result for geodesic generating functions is an algorithm for computing these geodesic generating functions.
Again, the algorithm is based on branching rules for words based on type.

When $\Gamma$ is link-regular, all types of the same cardinality can be clubbed together for the purpose of computing geodesic generating functions. By doing so, we recover a result of Antol\'in and Ciobanu \cite{1203.2752}  on the geodesic growth functions of link-regular graphs.

The formulae and algorithms described in this paper have been implemented using the Sage Mathematical Software \cite{sage}.
The code is available from \url{http://www.imsc.res.in/\~amri/growth/}.
This software allows the user to recover all the examples in \cite{1203.2752,MR1949695} with a few keystrokes. 
\section{Spherical Growth Functions}
\label{sec:type-an-element}

\begin{definition}
  \label{definition:type}
  The type of an element $x$ in $\gm$, $\raag$ or $\racg$ is the subset of $\m$ consisting of those indices $i$ such that $x$ is represented by a reduced word ending in $x_i$ or in $x_i^{-1}$.
  We write $\tau(x)$ for the type of $x$.
\end{definition}

In particular, the type of the identity element (which is represented by the empty word) is the empty set.

\begin{lemma}
  For each element $x$ in $\gm$, $\raag$ or $\racg$, the elements of $\tau(x)$ form a clique in $\Gamma$.
\end{lemma}
\begin{proof}
  This is clear for $\gm$, because all words representing the same element can be obtained from one another by a sequence of shuffles as in (\ref{eq:1}).
  If a word ending in $x_i$ can be changed into a word ending in $x_j$ using a sequence of shuffles, then at some stage $x_i$ has to be shuffled past $x_j$, so $i$ and $j$ must be joined by an edge in $\Gamma$.
  
  For $\raag$ and $\racg$, it turns out that all \emph{reduced} words representing the same element can be obtained from one another using a sequence of shuffles \cite[Theorem~3.9]{green_thesis}\footnote{For Coxeter groups, this is a well-known theorem of Tits \cite[Th\'eor\`eme~3]{Tits-words}.} and once again the lemma follows.
\end{proof}

\begin{lemma}
  \label{lemma:word-extn}
  For each element $x$ in $\gm$, $\raag$ or $\racg$, if $l(xx_i^{\pm 1})>l(x)$, then $\tau(xx_i^{\pm 1})$ is the unique maximal clique of $\tau(x)\cup \{i\}$ that contains $i$.
\end{lemma}
\begin{proof}
  Since $l(xx_i^{\pm 1})>l(x)$, a reduced word for $xx_i^{\pm 1}$ can be obtained by appending $x_i^{\pm 1}$ to a reduced word for $x$.
  Now $x_j^{\pm 1}$ can be shuffled to the end of $xx_i^{\pm 1}$ if and only if either $j = i$, or $x_j\in \tau(x)$ and $x_j$ can be shuffled past $x_i$.
  In other words, $\tau(xx_i^{\pm 1})$ consists of those elements of $\tau(x)\cup \{i\}$ which share an edge with the node $i$ of $\Gamma$, which is nothing but the unique maximal clique in $\tau(x)\cup \{i\}$ containing $i$.
\end{proof}
\begin{definition}
  Let $x$ be an element of $\gm$, $\raag$ or $\racg$ of positive length.
  Let $i = \max \tau(x)$.
  Then $x$ has a reduced word ending in $x_i^{\epsilon}$ where $\epsilon = \pm 1$.
  The element $x' = x x_i^{-\epsilon}$ is called the parent of $x$ and $x$ is called a child of $x'$.
\end{definition}
While each element has a unique parent, it can have several children.
The parent-child relationship for elements can be understood in terms of the following notion of branching for cliques (see Lemma~\ref{lemma:branches}):
\begin{definition}
  \label{definition:branching}
  Let $C$ and $C'$ be cliques in $\Gamma$.
  Say that $C'$ branches to $C$ (denoted  $C'\to C$) if the following conditions hold:
  \begin{enumerate}
  \item $C - C'$ is a singleton set (call its unique element $i$),
  \item $C$ is the maximal clique in $C'\cup \{i\}$ containing $i$, and
  \item $i>j$ for every $j\in C$.
  \end{enumerate}
  We write $C'\to C$.
\end{definition}
Note that if $i>j$ for every $j\in C'$ then the maximal clique containing $i$ in $C'\cup \{i\}$ is the unique branch of $C'$ with $C-C' = \{i\}$.
Otherwise, $C'$ has no branches with $C'-C = \{i\}$.
\begin{lemma}
  \label{lemma:branches}
  Let $C'$ be a clique of $\Gamma$.
  The children of an element of type $C'$ are described by the following rules:
  \begin{enumerate}
  \item If $x'\in \gm$ has type $C'$ then $x'$ has a unique child of type $C$ for every branch $C$ of $C'$, one child of type $C'$ and no other children.
  \item If $x'\in \raag$ has type $C'$ then $x'$ has two children of type $C$ for every branch $C$ of $C'$, one child of type $C'$ and no other children.
  \item If $x'\in \racg$ has type $C'$ then $x'$ has one child of type $C$ for every branch $C$ of $C'$ and no other children.
  \end{enumerate}
\end{lemma}
\begin{proof}
  The proofs of the three assertions are slight variations on the same theme:

  Suppose that $x'\in \gm$ has type $C'$.
  Let $j = \max C'$.
  Then $x'x_j$ is a child of $x'$ of type $C'$.
  If $i>j$, then by Lemma~\ref{lemma:word-extn}, $x'x_i$ is a child of $x'$ of type $C$ where $C$ is the branch of $C'$ with $C - C' = \{i\}$.
  For all other values of $i$, $x'x_i$ is not a child of $x'$.
  This proves the assertion for $\gm$.

  Suppose $x'\in \raag$ and $\tau(x') = C'$.
  Let $i = \max C'$.
  If $x'$ has a reduced word ending in $x_i$, then $xx_i$ is a child of $x$ of type $C'$.
  If $x'$ has a reduced word ending in $x_i^{-1}$, then $xx_i^{-1}$ is again child of $x$ of type $C'$.
  In either case, $x'$ has a unique child of type $C'$.
  On the other hand, if $i>j$ for every $j\in C'$, then $xx_i$ and $xx_i^{-1}$ are both children of type $C$ where $C$ is the branch of $C'$ with $C - C' = \{i\}$.
  For $i<j$, $x'x_i^{\pm 1}$ is not a child of $x'$.
  This proves the assertion for $\raag$.
  
  Suppose $x'\in \racg$ and $\tau(x') = C'$.
  Let $j = \max C'$.
  Since $x'$ has a reduced word ending in $x_j$, $l(xx_i)<l(x)$, as $x_j^2 = 1$.
  On the other hand, if $i>j$ for every $j\in C'$, then $xx_i$ is a child of $x'$ type $C$ where $C$ is the branch of $C'$ with $C - C' = \{i\}$.
  For $i<j$, $x'x_i$ is not a child of $x'$.
  This proves the assertion for $\racg$.
\end{proof}

For any clique $C$ in $\Gamma$, let $m_n(C)$, $a_n(C)$ and $c_n(C)$ denote the number of elements of length $n$ and type $C$ in $\gm$, $\raag$ and $\racg$ respectively.
Lemma~\ref{lemma:branches} implies that
\begin{align*}
  m_n(C) & = m_{n-1}(C) + \sum_{C'\to C} m_{n-1}(C'),\\
  a_n(C) & = a_{n-1}(C) + 2\sum_{C'\to C} a_{n-1}(C'),\\
  c_n(C) & = \sum_{C'\to C} c_{n-1}(C').
\end{align*}

Enumerate all the non-empty cliques of $\Gamma$ in some order: $C_1,C_2,\dotsc$.
Let $\vec m_n$, $\vec a_n$ and $\vec c_n$ be the column vectors whose $i$th entries are $m_n(C_i)$, $a_n(C_i)$ and $c_n(C_i)$ respectively.
By Example~\ref{example:small-len}, $\vec a_1 = 2\vec m_1 = 2\vec c_1$.

Let $B_0$ be the matrix whose $(i,j)$th entry is given by
\begin{equation*}
  B_0(i,j) =
  \begin{cases}
    1 & \text{if } C_j\to C_i,\\
    0 & \text{otherwise}.
  \end{cases}
\end{equation*}
Lemma~\ref{lemma:branches} can be expressed in matrix form as follows:
\begin{align*}
  \vec m_n & = (I + B_0)\vec m_{n-1},\\
  \vec a_n & = (I + 2B_0)\vec a_{n-1},\\
  \vec c_n & = B_0 \vec c_{n-1}.
\end{align*}
Iterating these identities gives
\begin{align*}
  \vec m_n & = (I + B_0)^{n-1} \vec m_1,\\
  \vec m_n & = (I + 2B_0)^{n-1} \vec a_1,\\
  \vec c_n & = B_0^{n-1} \vec c_1,
\end{align*}
giving rational expressions for vector-valued generating functions:
\begin{align}
  \label{eq:8}
  \sum_{n=1}^\infty \vec m_n t^n & = t[I-(I + B_0)t]^{-1} \vec m_1,\\
  \label{eq:9}
  \sum_{n=1}^\infty \vec m_n t^n & = t[I-(I + 2B_0)t]^{-1} \vec a_1,\\
  \label{eq:10}
  \sum_{n=1}^\infty \vec c_n t^n & = t[I-B_0t]^{-1} \vec c_1.
\end{align}
Furthermore, functional relationships between such generating functions are obtained:
\begin{align*}
  \sum_{n = 1}^\infty \vec m_n t^n  & = \sum_{n=1}^\infty  (I + B_0)^{n-1}t^n\vec m_1\\
  & = \sum_{n=1}^\infty \sum_{k=0}^\infty \binom{n-1}k B_0^k t^n \vec m_1\\
  & = \sum_{n=1}^\infty \sum_{k=1}^\infty \binom{n-1}{k-1} B_0^{k-1}t^n\vec m_1\\
  & = \sum_{k=1}^\infty B_0^{k-1} \sum_{n=1}^\infty \binom{n-1}{k-1}t^N\vec m_1\\
  & = \sum_{k=1}^\infty B_0^{k-1} \left(\frac t{1-t}\right)^k \vec m_1\\
  & = \sum_{k=1}^\infty \vec c_k \left(\frac t{1-t}\right)^k.
\end{align*}
Similarly, one shows
\begin{equation*}
  \sum_{n=1}^\infty \vec a_n t^n  = \sum_{k=1}^\infty \vec c_k \left(\frac{2t}{1-t}\right)^k.
\end{equation*}
For each clique $C$ in $\Gamma$ consider the generating functions
\begin{align*}
  \gm^C(t) & = \sum_{n=1}^\infty m_n(C)t^n,\\
  \raag^C(t) & = \sum_{n=1}^\infty a_n(C)t^n \text{ and}\\
  \racg^C(t) & = \sum_{n=1}^\infty c_n(C)t^n.
\end{align*}
Then the calculations above show that
\begin{align}
  \label{eq:functional1}
  \gm^C(t) & = \racg^C\left(\frac t{1-t}\right)\\
  \label{eq:functional2}
  \raag^C(t) & = \racg^C\left(\frac{2t}{1-t}\right).
\end{align}
Summing over all cliques $C$ in $\Gamma$ give the identities \eqref{eq:6} and \eqref{eq:7} from Section~\ref{sec:overview}.
Inverting the identity \eqref{eq:7} gives $\racg^C(t) = \gm^C(t/(1+t))$.
Substituting this into \eqref{eq:6} gives $\raag^C(t) = \gm^C(2t/(1+t))$.
Thus the expression \eqref{eq:5} for the spherical generating function of $\raag$ follows from the corresponding expression \eqref{eq:3} for $\gm$, and similarly, the identities \eqref{eq:refined-racg} and \eqref{eq:refined-raag} follow from \eqref{eq:Kratt}.

\section{Geodesic Growth Functions}
\label{sec:geod-growth-funct}
The parent-child relationship for words is simpler than for elements:
\begin{definition}
Let $w$ be a word in $\gm$, $\raag$ or $\racg$.
Then the word $w'$ obtained by removing the last letter of $w$ is called the parent of $w$.
The word $w$ is said to be a child of $w'$.
\end{definition}
The analog of Lemma~\ref{lemma:branches} for words uses a weaker notion of branching for cliques:
\begin{definition}
  \label{definition:weak-branching}
  Let $C$ and $C'$ be cliques in $\Gamma$.
  Say that $C'$ branches weakly to $C$ (denoted  $C'\to C$) if the following conditions hold:
  \begin{enumerate}
  \item $C - C'$ is a singleton set (call its unique element $i$),
  \item $C$ is the maximal clique in $C'\cup \{i\}$ containing $i$, and
  \end{enumerate}
  We write $C'\rightsquigarrow C$ if $C'$ branches weakly to $C$.
\end{definition}
The only difference between weak branching defined here and the notion of branching in Definition~\ref{definition:branching} is that the third condition of Definition~\ref{definition:branching} has been dropped.

The notion of type can be extended from elements of $\gm$, $\raag$ and $\racg$ to reduced words in these structures: the type of a word is just the type of the element that it represents.
\begin{lemma}
  \label{lemma:word-branching}
  Let $C'$ be a clique of $\Gamma$.
  The children of a word $w'$ of type $C'$ are described by the following rules:
  \begin{enumerate}
  \item If $w'\in \gm$ has type $C'$ then $w'$ has a unique child of type $C$ for every weak branch $C$ of $C'$, $|C'|$ children of type $C'$ and no other children.
  \item If $w'\in \raag$ has type $C'$ then $w'$ has two children of type $C$ for every weak branch $C$ of $C'$, $|C'|$ children of type $C'$ and no other children.
  \item If $w'\in \racg$ has type $C'$ then $w'$ has one child of type $C$ for every weak branch $C$ of $C'$ and no other children.
  \end{enumerate}
\end{lemma}
\begin{proof}
  The proof of these assertions is quite similar to that of Lemma~\ref{lemma:branches}, and is omitted.
\end{proof}
For each clique $C$, let $m^*_n(C)$, $a^*_n(C)$ and $c^*_n(C)$ denote the number of reduced words of length $n$ and type $C$ in $\gm$, $\raag$ and $\racg$ respectively.
Then Lemma~\ref{lemma:word-branching} tells us that
\begin{align*}
  m^*_n(C) & = |C|m^*_{n-1}(C) + \sum_{C'\rightsquigarrow C} m^*_{n-1}(C'),\\
  a^*_n(C) & = |C|a^*_{n-1}(C) + 2\sum_{C'\rightsquigarrow C} a^*_{n-1}(C'),\\
  c^*_n(C) & = \sum_{C'\rightsquigarrow C} c^*_{n-1}(C')
\end{align*}
As in Section~\ref{sec:type-an-element}, enumerate the non-empty cliques of $\Gamma$ as $C_1, C_2,\dotsc$.
Let $B_1$ be the matrix whose $(i,j)$th entry is given by
\begin{equation*}
  B_1(i,j) =
  \begin{cases}
    1 & \text{if } C_j\rightsquigarrow C_i,\\
    0 & \text{otherwise}.
  \end{cases}
\end{equation*}
Also, let $D$ denote the diagonal matrix whose $(i,i)$th entry is $|C_i|$.
Let $\vec m^*_n$, $\vec a^*_n$ and $\vec c^*_n$ denote the column vectors whose $i$th entries are $m^*_n(C_i)$, $a^*_n(C_i)$ and $c^*_n(C_i)$ respectively.
Then $\vec m^*_1 = \vec m_1$, $\vec a^*_1 = \vec a_1$ and $\vec c^*_1 = \vec c_1$.

Moreover, the above recurrences for these vectors can be written in matrix form as
\begin{align*}
  \vec m^*_n & = (D + B_1)\vec m^*_{n-1},\\
  \vec a^*_n & = (D + 2B_1)\vec a^*_{n-1},\\
  \vec c^*_n & = B_1 \vec c^*_{n-1}.
\end{align*}
giving, as in Section~\ref{sec:type-an-element}, rational expressions for vector-valued generating functions:
\begin{align}
  \label{eq:8}
  \sum_{n=1}^\infty \vec m^*_n t^n & = t[I-(D + B_1)t]^{-1} \vec m_1,\\
  \label{eq:9}
  \sum_{n=1}^\infty \vec m^*_n t^n & = t[I-(D + 2B_1)t]^{-1} \vec a_1,\\
  \label{eq:10}
  \sum_{n=1}^\infty \vec c^*_n t^n & = t[I-B_1t]^{-1} \vec c_1.
\end{align}
Adding them up gives a rational expression for the geodesic generating functions:
\begin{align*}
  \ggm & = 1 + \ones t[I-(D + B_1)t]^{-1} \vec m_1,\\
  \graag & = 1 + \ones t[I-(D + 2B_1)t]^{-1} \vec a_1,\\
  \gracg & = 1 + \ones t[I-B_1t]^{-1} \vec c_1,
\end{align*}
where $\ones$ denotes the all-ones row vector of dimension equal to the number of non-empty cliques in $\Gamma$.

Since $D$ and $B_1$ do not usually commute, it is not possible to take the binomial expansion of $(D + B_1)^n$ and obtain anything like the functional relations \eqref{eq:functional1} and \eqref{eq:functional2}.

\section{Link-Regular Graphs}
\label{sec:link-regular-graphs}
Recall that the link of a clique $C$ in $\Gamma$ is the set of all nodes in $\Gamma$ which are not in $C$, but which are connected by an edge to each of the nodes in $C$.
We write $\lk(C)$ for the link of $C$.
\begin{definition}[Link-regularity]
  The graph $\Gamma$ is said to be link regular if any two cliques of the same cardinality have links of the same cardinality.
\end{definition}
The computation of geodesic generating functions for link-regular graphs can be simplified considerably by clubbing together all the types which are cliques of the same cardinality.

\begin{lemma}
  \label{lemma:branching-reg}
  Let $\Gamma$ be a link-regular graph.
  If $C_1'$ and $C_2'$ are cliques of the same cardinality in $\Gamma$, then for each non-negative integer $r$, the number of $r$-cliques that are weak branches of $C_1'$ is equal to the number of $r$-cliques that are weak branches of $C_2'$.
\end{lemma}
\begin{proof}
  For each non-negative integer $r$, let  $L_r$ denote the cardinality of the link of any $r$-clique in $\Gamma$.
  Fix a $k$-clique $C$.
  For each subset $S\subset C$, define
  \begin{equation*}
    F_C(S) = \{x\in \Gamma - C \mid \lk(x)\cap C = S\}.
  \end{equation*}
  It follows from Definition~\ref{definition:weak-branching} that:
  \begin{equation}
    \label{eq:quote}
    \text{$C$ branches to $S\cup \{x\}$ weakly if and only if $x\in F_C(S)$.}
  \end{equation}

  Define
  \begin{equation*}
    G_C(S) = \{x\in \Gamma - C\mid \lk(x)\cap C \supset S\}.
  \end{equation*}
  Then $G_C(S)$ consists of all the points of $\Gamma - C$ which lie in $\lk(S)$.
  If $|S| = r$, then
  \begin{equation}
    \label{eq:G}
    |G_C(S)| = L_r - (k-r),
  \end{equation}
  which depends only on $k$ and $r$ (and not on the particular choices of cliques $C$ and $S$ of cardinality $k$ and $r$ respectively).
  Write $G_k(r)$ for $G_C(S)$.

  By the principle of inclusion and exclusion,
  \begin{align*}
    F_C(S) & = \sum_{T\supset S}^k (-1)^{|T - S|} G(T)\\
    & = \sum_{i=r}^k (-1)^{i-r}\binom{k-r}{i-r} G_k(i),
  \end{align*}
  which is again independent of $C$ and $S$, so long as their cardinalities are preserved.
  Write $F_k(r)$ for $F_C(S)$.

  By (\ref{eq:quote}), a $k$-clique weakly branches to $\binom kr F_k(r)$ many $r+1$-cliques, and the Lemma follows.
\end{proof}

Let $\bar m_n(k)$, $\bar a_n(k)$ and $\bar c_n(k)$ denote the numbers of words of length $n$ and having type of cardinality $k$.
Then combining Lemma~\ref{lemma:word-branching} with Lemma~\ref{lemma:branching-reg} gives:
\begin{align*}
  \bar m_n(i) & = [i + F_j(i-1)]m_{n-1}(j),\\
  \bar a_n(i) & = [i + 2F_j(i-1)]a_{n-1}(j),\\
  \bar c_n(i) & = F_j(i-1)c_{n-1}(j)
\end{align*}
for $n\geq 2$.

Let $d$ be the cardinality of the largest clique of $\Gamma$.
Let $\bar B_1$ be the $d\times d$ matrix whose $(i,j)$th entry is given by
\begin{equation*}
  \bar B_1(i,j) = F_j(i-1).
\end{equation*}
Let $\bar D$ be the diagonal matrix with diagonal entries $1, 2,\dotsc, d$.
Let $\bar{\vec m}_n$, $\bar{\vec a}_n$ and $\bar{\vec c}_n$ denote the column vectors of length $d$ whose $i$th coordinates are $m_n(i)$, $a_n(i)$ and $c_n(i)$ respectively, for $i = 1,\dotsc, d$. 
Then $\bar{\vec m}_1=\bar{\vec c}_1$ is the vector $(m,0,\dotsc)$ and let $\bar{\vec a}_1 = 2\bar{\vec m}_1$.

Then as in Section~\ref{sec:geod-growth-funct}, we obtain rational expressions for vector-valued generating functions:
\begin{align}
  \label{eq:8a}
  \sum_{n=1}^\infty \bar{\vec m}_n t^n & = t[I-(\bar D + \bar B_1)t]^{-1} \bar{\vec m}_1,\\
  \label{eq:9a}
  \sum_{n=1}^\infty \bar{\vec m}_n t^n & = t[I-(\bar D + 2\bar B_1)t]^{-1} \bar{\vec a}_1,\\
  \label{eq:10a}
  \sum_{n=1}^\infty \bar{\vec c}_n t^n & = t[I-\bar B_1t]^{-1} \bar{\vec c}_1.
\end{align}
In the above calculation of geodesic generating functions for link-regular graphs, the only information that we have used is the number $m$ of nodes in $\Gamma$ and the numbers $L_1,L_2,L_3,\dotsc$.
The information contained in this data is the same as the information contained in the clique polynomial of $\Gamma$ because of the identities
\begin{equation*}
  nc_n = c_{n-1}L_{n-1} \text{ for } n = 2, 3, 4, \dotsc.
\end{equation*}

Thus we recover Theorem~5.1 of Antol\'in and Ciobanu \cite{1203.2752}:
\begin{theorem}
  If two link-regular graphs $\Gamma_1$ and $\Gamma_2$ have the same clique polynomial, then there is an equality of associated geodesic generating functions:
  \begin{gather*}
    \ggm[\Gamma_1] = \ggm[\Gamma_2]\\
    \graag[\Gamma_1] = \graag[\Gamma_2]\\
    \gracg[\Gamma_1] = \gracg[\Gamma_2].
  \end{gather*}
\end{theorem}

\subsection*{Acknowledgments} J.S.A. thanks The Institute of Mathematical Sciences, Chennai for its hospitality.
\bibliographystyle{alpha}
\bibliography{references}
\end{document}